\documentclass{amsart}
\usepackage{graphicx} 
\usepackage{amsmath,amssymb,amsthm, mathrsfs}
\usepackage{boxedminipage}
\usepackage[margin=1in]{geometry}
\usepackage{mathpazo}
\usepackage{fancyhdr}
\usepackage[utf8]{inputenc}
\usepackage{color}
\usepackage{tikz}
\usetikzlibrary{positioning}
\usepackage{}

\newtheorem{theorem}{Theorem}[section]
\newtheorem{proposition}[theorem]{Proposition}
\newtheorem{lemma}[theorem]{Lemma}
\newtheorem{definition}[theorem]{Definition}

\newtheorem{varexample}[theorem]{Example}
\newtheorem{corollary}[theorem]{Corollary}

\newenvironment{example}{\begin{varexample}
\begin{normalfont}}{\end{normalfont}
\end{varexample}}

\newcommand{\rk}{\mathrm{rk}}

\title{Refined Brill-Noether Theory for Complete Graphs}
\author{Haruku Aono}
\author{Eric Burkholder}
\author{Owen Craig}
\author{Ketsile Dikobe}
\author{David Jensen}
\author{Ella Norris}

\begin{document} 
\bibliographystyle{alpha}

\begin{abstract}
The divisor theory of the complete graph $K_n$ is in many ways similar to that of a plane curve of degree $n$.  We compute the splitting types of all divisors on the complete graph $K_n$.  We see that the possible splitting types of divisors on $K_n$ exactly match the possible splitting types of line bundles on a smooth plane curve of degree $n$.  This generalizes the earlier result of Cori and Le Borgne computing the ranks of all divisors on $K_n$, and the earlier work of Cools and Panizzut analyzing the possible ranks of divisors of fixed degree on $K_n$.
\end{abstract}

\maketitle

\section{Introduction}
\label{Sec:Intro}

Brill-Noether theory is the study of line bundles on algebraic curves.  Two important invariants of a line bundle are its degree and its rank, and it is common to study the space of line bundles with fixed degree and rank on a given curve $C$.  In a famous series of results from the 1980's, it was shown that, if $C$ is sufficiently general, then these spaces of line bundles are smooth \cite{Gieseker82} of the expected dimension \cite{GriffithsHarris80}, and irreducible when this dimension is positive \cite{FultonLazarsfeld81} .  When $C$ is not special, however, the situation is more mysterious.

The Brill-Noether theory of plane curves has been studied since at least the late 19th century, when Max Noether considered the possible ranks of line bundles of fixed degree on a plane curve $C \subset \mathbb{P}^2$ of degree $n$ \cite{Noether82}.  This problem was solved simultaneously by Hartshorne \cite{Hartshorne86} and Ciliberto \cite{Ciliberto84} a century later.  More recently, Larson and Vemulapalli studied a more refined invariant of line bundles on plane curves, known as the \emph{splitting type} \cite{LarsonVemulapalli}.  Given a curve $C$, a line bundle $\mathcal{L}$ on $C$, and a map $\pi \colon C \to \mathbb{P}^1$ of degree $n$, the pushforward $\pi_* \mathcal{L}$ is a vector bundle of rank $n$ on $\mathbb{P}^1$.  Since every vector bundle on $\mathbb{P}^1$ is isomorphic to a direct sum of line bundles, there exists a sequence of integers $(e_1 , \ldots , e_n)$, unique up to permutation, such that $\pi_* \mathcal{L} \cong \oplus_{i=1}^n \mathcal{O}_{\mathbb{P}^1} (e_i)$.  The sequence $(e_1 , \ldots , e_n)$, called the \emph{splitting type} of the line bundle $\mathcal{L}$, has attracted a great deal of recent interest \cite{Larson21, CookPowellJensen22, CPJ22, LLV25}.  In the case of a plane curve, the map $\pi \colon C \to \mathbb{P}^1$ is given by projection from a point in $\mathbb{P}^2$ not on $C$, and the splitting type is independent of the choice of point.  In \cite{LarsonVemulapalli}, Larson and Vemulapalli show that, for a general plane curve $C$, the locus of line bundles with fixed splitting type is smooth, and compute its dimension.

A standard way to approach problems in Brill-Noether theory is via degeneration.  One degenerates to a singular curve and analyzes the limiting behavior of line bundles.    In \cite{Baker08}, Baker develops a theory of divisors on graphs that is useful in such degeneration arguments.  Briefly, a \emph{divisor} on a graph is a formal $\mathbb{Z}$-linear combination of vertices of $G$.  Baker defines an equivalence relation on divisors and a notion of rank.  He shows that if $\mathfrak{C}$ is a regular semistable model over a discrete valuation ring, with general fiber $C$ and whose special fiber has dual graph $G$, then there is a specialization map from the Picard group of $C$ to that of $G$, and the rank cannot increase under specialization.

When studying plane curves of degree $n$, there is a natural choice of singular curve -- the union of $n$ general lines in $\mathbb{P}^2$.  The dual graph of this singular curve is the complete graph $K_n$, which is the graph with $n$ vertices labeled $v_1 , \ldots , v_n$, and where every pair of vertices is adjacent.   In the special case where $\mathfrak{C}$ is a flat family of plane curves whose special fiber is a union of $n$ general lines, the class of a line specializes to the divisor $L = v_1 + v_2 + \cdots + v_n$ on the complete graph $K_n$.

In \cite{CoriLeBorgne}, Cori and Le Borgne provide a formula for the rank of any divisor on the complete graph $K_n$.  In \cite{CoolsPanizzut}, Cools and Panizzut use this to compute the possible ranks of divisors of fixed degree on $K_n$, providing a new proof of the Ciliberto-Hartshorne result for general plane curves.  The main result of this note is a formula for the splitting type of a divisor on $K_n$.  Because the splitting type of a divisor contains strictly more information than the rank and degree, this result generalizes the earlier work of both Cori-Le Borgne and Cools-Panizzut.  In order to state our result, we first need some terminology.

\begin{definition}
A divisor $D = \sum_{i=1}^n a_i v_i$ on the complete graph $K_n$ is \emph{concentrated} if, for every $i \leq n$, we have $\# \{ j \mid a_j - \min \{ a_k \} \leq i-1 \} \geq i$.
\end{definition}

As noted in \cite{CoriLeBorgne}, concentrated divisors with fixed minimum coefficient $\min \{ a_k \}$ are in bijection with parking functions.  In their paper, Cori and Le Borgne refer to a divisor as \emph{parking} if $\min \{ a_k \mid k \leq n-1 \} = 0$ and, for every $i \leq n-1$, we have $\# \{ j \leq n-1 \mid a_j \leq i-1 \} \geq i$.  Note that there is no condition on the coefficient $a_n$.  We use the new term ``concentrated'' to distinguish these divisors from the parking divisors of Cori and Le Borgne, emphasizing that the conditions hold for all coefficients, including $a_n$.  Our first result is that every divisor on $K_n$ can be put into a standard form.

\begin{proposition}
\label{Prop:Concentrated}
Every divisor on the complete graph $K_n$ is equivalent to a concentrated divisor.
\end{proposition}

\noindent Indeed, in Section~\ref{Sec:Reps}, we provide an algorithm for computing a concentrated divisor equivalent to a given divisor on $K_n$.

Cori and Le Borgne define a divisor $D = \sum_{i=1}^n a_i v_i$ on the complete graph $K_n$ ito be \emph{sorted} if $a_1 \leq a_2 \leq \cdots \leq a_{n-1}$.  Note again that there is no condition on $a_n$.  We say that $D$ is \emph{super sorted} if $a_1 \leq a_2 \leq \cdots \leq a_n$.  Given a divisor $D$ on $K_n$, we can always choose a permutation of the vertices so that $D$ is sorted or super sorted.  This simplifies many of our arguments.  Note that a super sorted divisor $D = \sum_{i=1}^n a_i v_i$ on the complete graph $K_n$ is concentrated if and only if $0 \leq a_i - a_1 \leq i-1$ for all $i$.

For divisors in this standard form, there is a simple formula for the splitting type.

\begin{theorem}
\label{Thm:MainThm}
Let $D = \sum_{i=1}^n a_i v_i$ be a super sorted, concentrated divisor on the complete graph $K_n$, and let $e_i = a_i - i + 1$.  Then $D$ has splitting type $(e_1, \ldots , e_n)$.
\end{theorem}

As a consequence, we see that the possible splitting types of divisors on $K_n$ exactly match the possible splitting types of line bundles on a smooth plane curve of degree $n$.

\begin{corollary}
\label{Cor:PossibleSplitting}
There exists a divisor of splitting type $(e_1 , \ldots , e_n)$ on $K_n$ if and only if, when arranged in decreasing order $e_1 \geq e_2 \geq \cdots \geq e_n$, we have $e_i \leq e_{i+1} + 1$ for all $i$.
\end{corollary}

\section*{Acknowledgements} This research was conducted as a project with the University of Kentucky Math Lab, supported by NSF DMS-2054135.

\section{Preliminaries}
\label{Sec:Prelim}

\subsection{Divisors on Graphs}
\label{Sec:Divisors}
Recall from the introduction that a \emph{divisor} on a graph $G$ is a formal $\mathbb{Z}$-linear combination of vertices of $G$.  We think of a divisor $D = \sum_{i=1}^n a_i v_i$ as a configuration of poker chips on the vertices of the graph, where the vertex $v_i$ has $a_i$ chips.  The \emph{degree} of a divisor $D = \sum_{i=1}^n a_i v_i$ is the integer $\mathrm{deg} (D) = \sum_{i=1}^n a_i$.  In other words, the degree of $D$ is the total number of chips on the graph.  Given a divisor $D$ and a vertex $v$, we may \emph{fire} $v$ to obtain a new divisor $D'$.  The chip firing move sends one chip from $v$ to each of its neighbors.  Thus, the divisor $D'$ has one more chip on each vertex adjacent to $v$, and $\mathrm{deg}(v)$ fewer chips at $v$.  In the particular case of the complete graph $K_n$, the divisor $D'$ has one more chip on every vertex other than $v$, and $n-1$ fewer chips at $v$.  We say that two divisors on a graph $G$ are \emph{equivalent} if one can be obtained from the other by a sequence of chip-firing moves.  Note that the divisor $L$ on $K_n$ from the introduction is equivalent to the divisor $nv_j$ for all $j$.

Given a subset $A$ of the vertices of $G$, if we fire each of the vertices in $A$ in any order, this results in sending one chip along each edge from $A$ to its complement.  We refer to this as firing the set $A$.  On the complete graph $K_n$, if the set $A$ has size $j$, then when fire $A$, each vertex in $A$ loses $n-j$ chips, and each vertex not in $A$ gains $j$ chips.

\subsection{Reduced Divisors}
\label{Sec:Reduced}
A divisor $D = \sum_{i=1}^n a_i v_i$ on a graph $G$ is \emph{effective} if $a_i \geq 0$ for all $i$.  We say that $D$ is \emph{effective away from} a vertex $v_j$ if $a_i \geq 0$ for all $i \neq j$.  The divisor $D$ is $v_j$-\emph{reduced} if $D$ is effective away from $v_j$, and firing any subset $A$ not containing $v_j$ results in a divisor that is not effective away from $v_j$.  On the complete graph $K_n$, there is a simple characterization of $v_n$-reduced divisors.

\begin{lemma} \cite[Lemma~5]{CoolsPanizzut}
\label{Lem:Reduced}
A divisor $D = \sum_{i=1}^n a_i v_i$ on $K_n$ is $v_n$-reduced if and only if it is a parking divisor in the sense of \cite{CoriLeBorgne}.
\end{lemma}

For any choice of vertex $v$, every divisor on $G$ is equivalent to a unique $v$-reduced divisor.  Given a divisor $D$ that is effective away from $v$, there is an algorithm for computing the $v$-reduced divisor equivalent to $D$, known as \emph{Dhar's Burining Algorithm}.  For the vertex $v_n$ on the complete graph $K_n$, this algorithm is simple to describe.  First, find the maximum value of $i$ such that $\# \{ j \leq n-1 \mid a_i \leq i-1 \} \leq i-1$.  If no such $i$ exists, then by Lemma~\ref{Lem:Reduced}, the divisor $D$ is $v_n$-reduced and we are done.  Otherwise, fire the set $A = \{ v_j \mid j \leq n-1 \text{ and } a_j \geq i \}$.  Each vertex in $A$ loses $n - \vert A \vert$ chips and each vertex not in $a$ gains $\vert A \vert$ chips.  Since each vertex in $A$ has at least $i$ chips and $n - \vert A \vert \leq i$, the resulting divisor $D'$ is still effective away from $v_n$.  We then replace $D$ with $D'$ and iterate this procedue until the resulting divisor is $v_n$-reduced.  Note that, each time we fire the set $A$, the total number of chips on vertices other than $v_n$ decreases.  Thus, this procedure terminates after finitely many steps.

\subsection{Ranks and Splitting Types}
\label{Sec:Rank}

In \cite{BakerNorine07}, Baker and Norine define the rank of a divisor on a graph.

\begin{definition}
Let $D$ be a divisor on a graph $G$.  If $D$ is not equivalent to an effective divisor, we say that it has \emph{rank} $-1$.  Otherwise, we define the \emph{rank} of $D$ to be the maximum integer $r$ such that, for all effective divisors $E$ of degree $r$, $D-E$ is equivalent to an effective divisor.
\end{definition}

The \emph{canonical divisor} of a graph $G$ is the divisor $K_G = \sum_{i=i}^n (\mathrm{val}(v_i) - 2)v_i$.  On the complete graph $K_n$, the canonical divisor $K$ is equal to $(n-3)L$.  In \cite{BakerNorine07}, Baker and Norine prove an analogue of the Riemann-Roch Theorem for graphs.

\begin{theorem} \cite[Theorem~1.12]{BakerNorine07}
Let $G$ be a graph with first Betti number $g$, and let $D$ be a divisor on $G$.  Then
\[
\rk(D) - \rk(K_G - D) = \deg(D) - g + 1.
\]
\end{theorem}

\noindent As a consequence, note that if $D-K_G$ is effective and nontrivial, then $\rk(D) = \deg(D) - g$. 

The main result of \cite{CoriLeBorgne} is a formula for the rank of a divisor on a complete graph.

\begin{theorem} \label{Thm:Rank} \cite[Theorem~12]{CoriLeBorgne}\footnote{Note that there is a typo in the formula appearing on page 3 of the published version of \cite{CoriLeBorgne}.  The correct formula for the rank can be found on page 19.}
Let $D$ be a sorted, $v_n$-reduced divisor on the complete graph $K_n$.  Let $q$ and $r$ be the unique integers such that $a_n + 1 = q(n-1) + r$ and $0 \leq r \leq n-2$.  We define $\chi (P)$ to be 1 if the proposition $P$ is true and 0 if it is false.  Then
\[
\rk (D) = \Big( \sum_{i=1}^{n-1} \max \{ 0, q - i + 1 + a_i + \chi (i \leq r)  \} \Big) - 1.
\]
\end{theorem}

The goal of this paper is to compute the splitting type of divisors on the complete graph.  Note that the splitting type of a line bundle $\mathcal{L}$ is completely determined by the ranks of the line bundles $\mathcal{L} \otimes \mathcal{O}_{\mathbb{P}^2}(k)$ for all integers $k$.  More precisely, 
\begin{align*}
h^0 (C, \mathcal{L} \otimes \mathcal{O}_{\mathbb{P}^2}(k)) &= h^0 (\mathbb{P}^1, \pi_* \mathcal{L} \otimes \mathcal{O}_{\mathbb{P}^1} (k)) \\
&= h^0 (\mathbb{P}^1, \oplus_{i=1}^n \mathcal{O}_{\mathbb{P}^1} (e_i + k)) \\
&= \sum_{i=1}^n  h^0 (\mathbb{P}^1, \mathcal{O}_{\mathbb{P}^1} (e_i + k)) \\
&= \sum_{i=1}^n \max \{ 0 , e_i + k + 1 \}.
\end{align*}
For this reason, we define the splitting type of a divisor on $K_n$ as follows.

\begin{definition}
Let $e_1 , \ldots , e_n$ be integers.  A divisor $D$ on the complete graph $K_n$ has \emph{splitting type} $(e_1 , \ldots , e_n)$ if
\[
\rk (D + kL) = \Big( \sum_{i=1}^n \max \{ 0 , e_i + k + 1 \} \Big) - 1 \text{ for all } k \in \mathbb{Z} .
\]
\end{definition}

\section{Canonical Representatives of Divisors on Complete Graphs}
\label{Sec:Reps}

We now prove Proposition~\ref{Prop:Concentrated}.


\begin{proof}[Proof of Proposition~\ref{Prop:Concentrated}]
Let $D$ be a divisor on the complete graph $K_n$.  We provide an algorithm for producing a concentrated divisor equivalent to $D$.  First, by running Dhar's Burning Algorithm, we may reduce to the case where $D$ is $v_n$-reduced.  It follows that $\min \{ a_i \mid i \leq n-1 \} = 0$ and, for every $i \leq n-1$, we have $\# \{ j \leq n-1 \mid a_j \leq i-1 \} \geq i$.  Now, let $m$ be the maximum integer such that $a_n - mn \geq 0$, and let $D' = D - mnv_n$.  Note that $D' + mL$ is equivalent to $D$.  Our goal is to show that $D' + mL$ is concentrated.  Note that $D' + mL$ is concentrated if and only if $D'$ is, so it suffices to show that $D'$ is concentrated.  By construction, $0 \leq a_n \leq n-1$, hence $\# \{ j \mid a_j \leq n-1 \} = n$.  For $i \leq n-1$, we have 
\[
\# \{ j \mid a_j \leq i-1 \} \geq \# \{ j \leq n-1 \mid a_j \leq i-1 \} \geq i,
\]
hence $D'$ is concentrated.
\end{proof}

Note that the concentrated divisor equivalent to $D$ is not unique.  For example, the divisors $D = (n-1)v_n$ and $D' = v_1 + \cdots + v_{n-1}$ are equivalent, and both are concentrated.

\begin{example}
\label{Ex:Algorithm}
Figure~\ref{Fig:Algorithm} depicts the algorithm described in the proof of Proposition~\ref{Prop:Concentrated}.  In the upper left, we see the complete graph $K_5$, with its vertices labeled $v_1, \ldots, v_5$.  In the middle figure of the top row, we see a divisor $D$ on this graph.  This divisor is not concentrated, since $6-0 = 6 > 4$.  This divisor is also not $v_5$-reduced, because there are 6 chips on $v_4$.  Applying Dhar's Burning Algorithm, we first fire $v_4$ to obtain the divisor on the top right.  (Note that this divisor is concentrated, but the algorithm described in the proof of Proposition~\ref{Prop:Concentrated} has not yet terminated.)  This divisor is still not $v_5$-reduced, because there is at least 1 chip on every vertex.  We then fire the set $A = \{ v_1, v_2, v_3 , v_4 \}$ to obtain the divisor on the bottom left, which is $v_5$-reduced by Lemma~\ref{Lem:Reduced}.  This divisor has 6 chips at $v_5$, thus $m=1$ is the largest integer such that $6-5m \geq 0$.  The divisor $D - 1 \cdot L$ is equivalent to the divisor in the middle of the bottom row.  Finally, adding the divisor $L$ to this, we obtain the concentrated divisor on the bottom right.  Thus we see that $D$ is equivalent to a concentrated divisor. 

\begin{figure}[h]
    \begin{tikzpicture}
    \filldraw (0,1) circle (2pt);
    \filldraw (0.95,0.31) circle (2pt);
    \filldraw (-0.95,0.31) circle (2pt);
    \filldraw (0.59,-0.81) circle (2pt);
    \filldraw (-0.59,-0.81) circle (2pt);
    \draw (0,1)--(0.95,0.31);
    \draw (0,1)--(-0.95,0.31);
    \draw (0,1)--(0.59,-0.81);
    \draw (0,1)--(-0.59,-0.81);
    \draw (0.95,0.31)--(-0.95,0.31);
    \draw (0.95,0.31)--(0.59,-0.81);
    \draw (0.95,0.31)--(-0.59,-0.81);
    \draw (-0.95,0.31)--(0.59,-0.81);
    \draw (-0.95,0.31)--(-0.59,-0.81);
    \draw (0.59,-0.81)--(-0.59,-0.81);
    \node at (0,1.3) {$v_5$};
    \node at (1.2,0.31) {$v_1$};
    \node at (0.9,-1) {$v_2$};
    \node at (-0.9,-1) {$v_3$};
    \node at (-1.2,0.31) {$v_4$};

    \filldraw (3,1) circle (2pt);
    \filldraw (3.95,0.31) circle (2pt);
    \filldraw (2.05,0.31) circle (2pt);
    \filldraw (3.59,-0.81) circle (2pt);
    \filldraw (2.41,-0.81) circle (2pt);
    \draw (3,1)--(3.95,0.31);
    \draw (3,1)--(2.05,0.31);
    \draw (3,1)--(3.59,-0.81);
    \draw (3,1)--(2.41,-0.81);
    \draw (3.95,0.31)--(2.05,0.31);
    \draw (3.95,0.31)--(3.59,-0.81);
    \draw (3.95,0.31)--(2.41,-0.81);
    \draw (2.05,0.31)--(3.59,-0.81);
    \draw (2.05,0.31)--(2.41,-0.81);
    \draw (3.59,-0.81)--(2.41,-0.81);
    \node at (3,1.3) {$1$};
    \node at (4.2,0.31) {$0$};
    \node at (3.9,-1) {$2$};
    \node at (2.1,-1) {$0$};
    \node at (1.8,0.31) {$6$};

    \filldraw (6,1) circle (2pt);
    \filldraw (6.95,0.31) circle (2pt);
    \filldraw (5.05,0.31) circle (2pt);
    \filldraw (6.59,-0.81) circle (2pt);
    \filldraw (5.41,-0.81) circle (2pt);
    \draw (6,1)--(6.95,0.31);
    \draw (6,1)--(5.05,0.31);
    \draw (6,1)--(6.59,-0.81);
    \draw (6,1)--(5.41,-0.81);
    \draw (6.95,0.31)--(5.05,0.31);
    \draw (6.95,0.31)--(6.59,-0.81);
    \draw (6.95,0.31)--(5.41,-0.81);
    \draw (5.05,0.31)--(6.59,-0.81);
    \draw (5.05,0.31)--(5.41,-0.81);
    \draw (6.59,-0.81)--(5.41,-0.81);
    \node at (6,1.3) {$2$};
    \node at (7.2,0.31) {$1$};
    \node at (6.9,-1) {$3$};
    \node at (5.1,-1) {$1$};
    \node at (4.8,0.31) {$2$};

    \filldraw (0,-2) circle (2pt);
    \filldraw (0.95,-2.69) circle (2pt);
    \filldraw (-0.95,-2.69) circle (2pt);
    \filldraw (0.59,-3.81) circle (2pt);
    \filldraw (-0.59,-3.81) circle (2pt);
    \draw (0,-2)--(0.95,-2.69);
    \draw (0,-2)--(-0.95,-2.69);
    \draw (0,-2)--(0.59,-3.81);
    \draw (0,-2)--(-0.59,-3.81);
    \draw (0.95,-2.69)--(-0.95,-2.69);
    \draw (0.95,-2.69)--(0.59,-3.81);
    \draw (0.95,-2.69)--(-0.59,-3.81);
    \draw (-0.95,-2.69)--(0.59,-3.81);
    \draw (-0.95,-2.69)--(-0.59,-3.81);
    \draw (0.59,-3.81)--(-0.59,-3.81);
    \node at (0,-1.7) {$6$};
    \node at (1.2,-2.69) {$0$};
    \node at (0.9,-4) {$2$};
    \node at (-0.9,-4) {$0$};
    \node at (-1.2,-2.69) {$1$};

    \filldraw (3,-2) circle (2pt);
    \filldraw (3.95,-2.69) circle (2pt);
    \filldraw (2.05,-2.69) circle (2pt);
    \filldraw (3.59,-3.81) circle (2pt);
    \filldraw (2.41,-3.81) circle (2pt);
    \draw (3,-2)--(3.95,-2.69);
    \draw (3,-2)--(2.05,-2.69);
    \draw (3,-2)--(3.59,-3.81);
    \draw (3,-2)--(2.41,-3.81);
    \draw (3.95,-2.69)--(2.05,-2.69);
    \draw (3.95,-2.69)--(3.59,-3.81);
    \draw (3.95,-2.69)--(2.41,-3.81);
    \draw (2.05,-2.69)--(3.59,-3.81);
    \draw (2.05,-2.69)--(2.41,-3.81);
    \draw (3.59,-3.81)--(2.41,-3.81);
    \node at (3,-1.7) {$1$};
    \node at (4.2,-2.69) {$0$};
    \node at (3.9,-4) {$2$};
    \node at (2.1,-4) {$0$};
    \node at (1.8,-2.69) {$1$};

    \filldraw (6,-2) circle (2pt);
    \filldraw (6.95,-2.69) circle (2pt);
    \filldraw (5.05,-2.69) circle (2pt);
    \filldraw (6.59,-3.81) circle (2pt);
    \filldraw (5.41,-3.81) circle (2pt);
    \draw (6,-2)--(6.95,-2.69);
    \draw (6,-2)--(5.05,-2.69);
    \draw (6,-2)--(6.59,-3.81);
    \draw (6,-2)--(5.41,-3.81);
    \draw (6.95,-2.69)--(5.05,-2.69);
    \draw (6.95,-2.69)--(6.59,-3.81);
    \draw (6.95,-2.69)--(5.41,-3.81);
    \draw (5.05,-2.69)--(6.59,-3.81);
    \draw (5.05,-2.69)--(5.41,-3.81);
    \draw (6.59,-3.81)--(5.41,-3.81);
    \node at (6,-1.7) {$2$};
    \node at (7.2,-2.69) {$1$};
    \node at (6.9,-4) {$3$};
    \node at (5.1,-4) {$1$};
    \node at (4.8,-2.69) {$2$};

    \end{tikzpicture}
    \caption{Finding a concentrated divisor.}
    \label{Fig:Algorithm}
\end{figure}
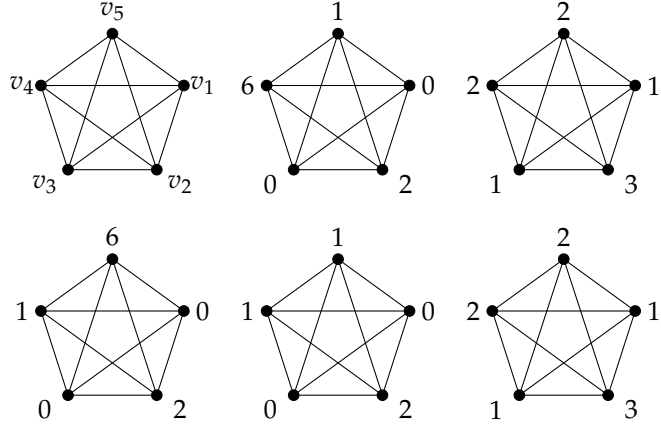

\end{example}

\section{Splitting Types of Divisors on Complete Graphs}
\label{Sec:Main}

We now prove the main theorem.  The proof is essentially a calculation using Theorem~\ref{Thm:Rank}, which is a bit lengthy due to the number of cases.


\begin{proof}[Proof of Theorem~\ref{Thm:MainThm}]
Note that, if $D$ has splitting type $(e_1 , \ldots , e_n)$, then $D + mL$ has splitting type $(e_1 + m, \ldots, e_n + m)$.  We may therefore further reduce to the case where $a_1 = 0$.  In this case, by Lemma~\ref{Lem:Reduced}, $D + knv_n$ is $v_n$-reduced for all integers $k$.  We prove, by induction on $k$, that 
\begin{align}
\label{eq:rank}
\rk (D + knv_n) = \Big( \sum_{i=1}^n \max \{ 0 , a_i - i + k + 2 \} \Big) - 1
\end{align}
for all integers $k$.  For the base case, suppose that $k \leq -1$.  Since $D + knv_n$ is $v_n$-reduced and has a negative number of chips at $v_n$, we see that $\rk (D + knv_n) = -1$.  On the other hand, since $a_i \leq i-1$ for all $i$, we have $a_i - i + k +2 \leq 0$ for all $i$, hence \eqref{eq:rank} holds.

For the inductive step, it suffices to prove that
\begin{align}
\label{eq:rankjump}
\rk (D + knv_n) = \rk(D + (k-1)nv_n) + \# \{ i \mid a_i - i +1  \geq -k \}.
\end{align}
The divisor $D + knv_n$ is sorted and $v_n$-reduced for all integers $k$, so we may apply Theorem~\ref{Thm:Rank}.  We let $q, q', r, r'$ be the unique integers such that
\begin{align*}
a_n + (k-1)n + 1 &= q(n-1) + r \\
a_n + kn + 1 &= q'(n-1) + r' \\
& 0 \leq r, r' \leq n-2 .
\end{align*}
By Theorem~\ref{Thm:Rank}, we have
\begin{align}
\label{eq:inductivestep}
\rk (D + knv_n) - \rk (D + (k-1)nv_n) &= \sum_{i=1}^{n-1} \Big( \max \{ 0, q' - i + 1 + a_i + \chi (i \leq r')  \} \\ 
&- \max \{ 0, q - i + 1 + a_i + \chi (i \leq r)  \} \Big). \notag
\end{align}
We break this into several cases.  

\textbf{Case 1:} First, consider the case where $0 \leq k \leq n - a_n - 3$.  Note that the value $n$ does not contribute to the sum \eqref{eq:rankjump}.  In this case, $q'=k$, $r = a_n + k$, $r' = r+1$, and $q' = q+1$.  Thus, \eqref{eq:inductivestep} becomes
\begin{align*}
\rk (D + knv_n) - \rk (D + (k-1)nv_n) &= \sum_{i=1}^{n-1} \Big( \max \{ 0, a_i - i + 1 + k + \chi (i \leq r+1)  \} \\ 
&- \max \{ 0, a_i - i + k + \chi (i \leq r)  \} \Big).
\end{align*}

If $i \leq r$, then
\begin{align*}
\max \{ 0, a_i - i + 1 + k + \chi (i \leq r+1)  \} - \max \{ 0, a_i - i + k + \chi (i \leq r)  \} \\
= \max \{ 0, a_i - i + 2 + k  \} - \max \{ 0, a_i - i + 1 + k  \} \\
= \begin{cases}
0 & \mbox{ if $a_i - i + 1 \leq -k - 1$} \\
1 & \mbox{ if $a_i - i + 1 \geq -k $.} 
\end{cases}
\end{align*}
Thus, the value $i$ contributes 1 to the sum \eqref{eq:inductivestep} if and only if $a_i - i + 1 \geq -k$, if and only if it contributes 1 to the sum \eqref{eq:rankjump}.

If $i \geq r+2$, then since $D$ is super sorted, we have $a_i - i + 1 \leq a_n - i + 1 \leq a_n - r - 1 = -k-1$.  Thus, the value $i$ does not contribute to the sum \eqref{eq:rankjump}.  Now, we have
\begin{align*}
\max \{ 0, a_i - i + 1 + k + \chi (i \leq r+1)  \} - \max \{ 0, a_i - i + k + \chi (i \leq r)  \} \\
= \max \{ 0, a_i - i + 1 + k  \} - \max \{ 0, a_i - i + k  \} \\
= \begin{cases}
0 & \mbox{ if $a_i - i + 1 \leq -k$} \\
1 & \mbox{ if $a_i - i + 1 \geq -k +1$.} 
\end{cases}
\end{align*}
Thus, the value $i$ does not contribute to the sum \eqref{eq:inductivestep}.

If $i = r+1$, then since $D$ is super sorted, we have $a_i - i + 1 = a_i - r \leq a_n - r = -k$.  Thus, the value $i$ contributes 1 to the sum \eqref{eq:rankjump} if and only if $a_i - i +1 = -k$.  Now, we have
\begin{align*}
\max \{ 0, a_i - i + 1 + k + \chi (i \leq r+1)  \} - \max \{ 0, a_i - i + k + \chi (i \leq r)  \} \\
= \max \{ 0, a_i - i + 2 + k  \} - \max \{ 0, a_i - i + k  \} \\
= \begin{cases}
0 & \mbox{ if $a_i - i + 1 \leq -k-1$} \\
1 & \mbox{ if $a_i - i + 1 = -k$}\\
2 & \mbox{ if $a_i - i + 1 \geq -k +1$.} 
\end{cases}
\end{align*}
Thus, the value $i$ contributes 1 to the sum \eqref{eq:inductivestep} if and only if $a_i - i + 1 = -k$.  Putting this all together, we see that \eqref{eq:rankjump} holds in this case.

\textbf{Case 2:}  Next, consider the case where $k = n - a_n - 2 \geq 0$.  As before, the value $n$ does not contribute to the sum \eqref{eq:rankjump}.  In this case, $q = k-1$, $q' = k+1$, $r = n-2$, and $r' = 0$.  Thus, \eqref{eq:inductivestep} becomes
\begin{align*}
\label{eq:inductivestep}
\rk (D + knv_n) - \rk (D + (k-1)nv_n) &= \sum_{i=1}^{n-1} \max \{ 0, k - i + 2 + a_i  \} \\ 
&- \sum_{i=1}^{n-2} \Big( \max \{ 0, a_i - i + 1 + k  \} \Big) - \max \{ 0, a_{n-1} - n + 1 + k \}.
\end{align*}
If $i \leq n-2$, we have
\begin{align*}
 \max \{ 0, a_i - i + 2 + k  \} - \max \{ 0, a_i - i + 1 + k  \} \\
=  \begin{cases}
0 & \mbox{ if $a_i - i + 1 \leq -k-1$} \\
1 & \mbox{ if $a_i - i + 1 \geq -k$.}
\end{cases}
\end{align*}
Thus, the value $i$ contriubtes 1 to the sum \eqref{eq:inductivestep} if and only if $a_i - i + 1 \geq -k$, if and only if it contributes 1 to the sum \eqref{eq:rankjump}.

If $i = n-1$, then since $D$ is super sorted, we have $a_{n-1} - (n-1) + 1 \leq a_n - n + 2 = -k$.  We have
\begin{align*}
 \max \{ 0, a_{n-1} - (n-1) + 2 + k  \} - \max \{ 0, a_{n-1} - n + 1 + k  \} \\
=  \begin{cases}
0 & \mbox{ if $a_{n-1} - (n-1) + 1 \leq -k-1$} \\
1 & \mbox{ if $a_{n-1} - (n-1) + 1 = -k$} \\
2 & \mbox{ if $a_{n-1} - (n-1) + 1 \geq -k+1$.}
\end{cases}
\end{align*}
Thus, the value $i = n-1$ contributes 1 to the sum \eqref{eq:inductivestep} if and only if $a_{n-1} - (n-1) + 1 = -k$, if and only if it contributes 1 to the sum \eqref{eq:rankjump}.  Putting this all together, we see that \eqref{eq:rankjump} holds in this case.

\textbf{Case 3:}  Next, consider the case where $0 < n - a_n - 1 \leq k \leq n-3$.  Here, the value $n$ contributes 1 to the sum \eqref{eq:rankjump}.  In this case, $q'=k+1$, $r = a_n + k - (n-1)$, $r' = r+1$, and $q' = q+1$.  Thus, \eqref{eq:inductivestep} becomes
\begin{align*}
\rk (D + knv_n) - \rk (D + (k-1)nv_n) &= \sum_{i=1}^{n-1} \Big( \max \{ 0, a_i - i + 2 + k + \chi (i \leq r+1)  \} \\ 
&- \max \{ 0, a_i - i + 1 + k + \chi (i \leq r)  \} \Big).
\end{align*}

If $i \leq r$, then $a_i - i + 1 \geq -i + 1 \geq -r+1 = n-k-a_n \geq -k+1$.  Thus, $i$ contributes 1 to the sum \eqref{eq:rankjump}.  We have 
\begin{align*}
\max \{ 0, a_i - i + 2 + k + \chi (i \leq r+1)  \} - \max \{ 0, a_i - i + 1 + k + \chi (i \leq r)  \} \\
= \max \{ 0, a_i - i + 3 + k  \} - \max \{ 0, a_i - i + 2 + k  \} \\
= \begin{cases}
0 & \mbox{ if $a_i - i + 1 \leq -k - 2$} \\
1 & \mbox{ if $a_i - i + 1 \geq -k-1 $.} 
\end{cases}
\end{align*}
Thus, the value $i$ contributes 1 to the sum \eqref{eq:inductivestep} as well.

If $i \geq r+2$, then we have
\begin{align*}
\max \{ 0, a_i - i + 2 + k + \chi (i \leq r+1)  \} - \max \{ 0, a_i - i + 1 + k + \chi (i \leq r)  \} \\
= \max \{ 0, a_i - i + 2 + k  \} - \max \{ 0, a_i - i + 1 + k  \} \\
= \begin{cases}
0 & \mbox{ if $a_i - i + 1 \leq -k - 1$} \\
1 & \mbox{ if $a_i - i + 1 \geq -k$.} 
\end{cases}
\end{align*}
Thus, the value $i$ contributes to the sum \eqref{eq:inductivestep} if and only if $a_i - i + 1 \geq -k$, if and only if it contributes 1 to the sum \eqref{eq:rankjump}.

If $i = r+1$, then $a_i - i + 1 = a_i - r \geq -r = n-k-a_n - 1 \geq -k$.  Thus, $i$ contributes 1 to the sum \eqref{eq:rankjump}.  We have
\begin{align*}
\max \{ 0, a_i - i + 2 + k + \chi (i \leq r+1)  \} - \max \{ 0, a_i - i + 1 + k + \chi (i \leq r)  \} \\
= \max \{ 0, a_i - i + 3 + k  \} - \max \{ 0, a_i - i + 1 + k  \} \\
= \begin{cases}
0 & \mbox{ if $a_i - i + 1 \leq -k - 2$} \\
1 & \mbox{ if $a_i - i + 1 = -k-1$} \\
2 & \mbox{ if $a_i - i + 1 \geq - k$.} 
\end{cases}
\end{align*}
Thus, $i$ contributes 2 to the sum \eqref{eq:inductivestep}.  Since both $i$ and $n$ contribute 1 to the sum \eqref{eq:rankjump}, we see that \eqref{eq:rankjump} holds in this case.

\textbf{Case 4:}  Finally, consider the case where $k \geq n-2$.  Note that the canonical divisor $K$ is equal to $(n-3)L$.  Since $D$ is effective, we see that $D + kL - K$ is effective and nontrivial.  It follows from Riemann-Roch that 
\begin{align*}
\rk (D + knv_n) = kn - {{n-1}\choose{2}} +  \sum_{i=1}^n a_i &= \Big( \sum_{i=1}^n (a_i - i + k + 2) \Big) -1  \\
&=  \Big( \sum_{i=1}^n \max \{ 0 , a_i - i + k + 2 \} \Big) - 1.
\end{align*}
\end{proof}

\begin{example}
Consider the divisor $D$ pictured on the left in Figure~\ref{Fig:Splitting}.  In Example~\ref{Ex:Algorithm}, we showed that $D$ is equivalent to the concentrated divisor $D'$ pictured on the right in Figure~\ref{Fig:Splitting}.  By Theorem~\ref{Thm:MainThm}, we have
\begin{align*}
e_1 &= 1 -1 + 1 = 1 \\
e_2 &= 1 -2 + 1 = 0 \\
e_3 &= 2 - 3 + 1 = 0 \\
e_4 &= 2 - 4 + 1 = -1 \\
e_5 &= 3 - 5 + 1 = -1. 
\end{align*}
Thus, the splitting type of the divisor $D$ is $(1,0,0,-1,-1)$.  Since all of the terms in the splitting type are greater than $-2$, the divisor $D$ is nonspecial.  In other words, the rank of $D$ is 3, which is equal to that of a general divisor of degree 9 on a plane quintic.  The splitting type of $D$, however, is not equal to that of a general divisor of degree 9 on a plane quintic, which is $(0,0,0,0,-1)$.  This can be seen from the fact that $D-L$ is effective, which is false for a general divisor of degree 9.

\begin{figure}[h]
    \begin{tikzpicture}
    \filldraw (0,1) circle (2pt);
    \filldraw (0.95,0.31) circle (2pt);
    \filldraw (-0.95,0.31) circle (2pt);
    \filldraw (0.59,-0.81) circle (2pt);
    \filldraw (-0.59,-0.81) circle (2pt);
    \draw (0,1)--(0.95,0.31);
    \draw (0,1)--(-0.95,0.31);
    \draw (0,1)--(0.59,-0.81);
    \draw (0,1)--(-0.59,-0.81);
    \draw (0.95,0.31)--(-0.95,0.31);
    \draw (0.95,0.31)--(0.59,-0.81);
    \draw (0.95,0.31)--(-0.59,-0.81);
    \draw (-0.95,0.31)--(0.59,-0.81);
    \draw (-0.95,0.31)--(-0.59,-0.81);
    \draw (0.59,-0.81)--(-0.59,-0.81);
    \node at (0,1.3) {$1$};
    \node at (1.2,0.31) {$0$};
    \node at (0.9,-1) {$2$};
    \node at (-0.9,-1) {$0$};
    \node at (-1.2,0.31) {$6$};

    \filldraw (5,1) circle (2pt);
    \filldraw (5.95,0.31) circle (2pt);
    \filldraw (4.05,0.31) circle (2pt);
    \filldraw (5.59,-0.81) circle (2pt);
    \filldraw (4.41,-0.81) circle (2pt);
    \draw (5,1)--(5.95,0.31);
    \draw (5,1)--(4.05,0.31);
    \draw (5,1)--(5.59,-0.81);
    \draw (5,1)--(4.41,-0.81);
    \draw (5.95,0.31)--(4.05,0.31);
    \draw (5.95,0.31)--(5.59,-0.81);
    \draw (5.95,0.31)--(4.41,-0.81);
    \draw (4.05,0.31)--(5.59,-0.81);
    \draw (4.05,0.31)--(4.41,-0.81);
    \draw (5.59,-0.81)--(4.41,-0.81);
    \node at (5,1.3) {$2$};
    \node at (6.2,0.31) {$1$};
    \node at (5.9,-1) {$3$};
    \node at (4.1,-1) {$1$};
    \node at (3.8,0.31) {$2$};

    \end{tikzpicture}
    \caption{A divisor $D$ (left) and its associated concentrated divisor $D'$ (right).}
    \label{Fig:Splitting}
\end{figure}
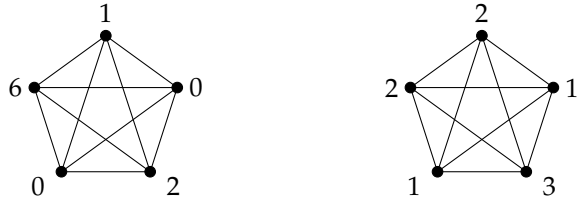
\end{example}

Finally, we prove Corollary~\ref{Cor:PossibleSplitting}.

\begin{proof}[Proof of Corollary~\ref{Cor:PossibleSplitting}]
First, let $D = \sum_{i=1}^n a_i v_i$ be a divisor on the complete graph, and let $(e_1, \ldots , e_n)$ be its splitting type.  We show that, when arranged in decreasing order $e_1 \geq \cdots \geq e_n$, we have $e_i \leq e_{i+1} + 1$ for all $i$.  By Proposition~\ref{Prop:Concentrated}, we may assume that $D$ is concentrated, and by permuting the vertices, we may assume that $D$ is super sorted.  Then $e_i = a_i - i + 1$ for all $i$.  Since $D$ is concentrated and super sorted, $e_i \leq e_1$ for all $i$.  Because the sequence $(e_1, \ldots, e_n)$ obtains its maximum at $e_1$, it suffices to show that if $e_i < e_{i-1}$, then $e_{i-1} = e_i + 1$.  But since $a_i \geq a_{i-1}$, if $e_i < e_{i-1}$, we see that $a_i = a_{i-1}$ and $e_{i-1} = e_i + 1$.  

For the converse, let $(e_1, \ldots , e_n)$ be a splitting type with $e_i \geq e_{i+1} \geq e_i - 1$ for all $i$.  Now, let $a_i = e_i + i - 1$ for all $i$, and let $D = \sum_{i=1}^n a_i v_i$.  By assumption, we have $e_{i+1} \geq e_i -1$ for all $i$, so $a_{i+1} \geq a_i$ for all $i$, hence $D$ is super sorted.  Also by assumption, we have $e_i \leq e_1$ for all $i$, so $a_i - a_1 \leq i-1$ for all $i$, hence $D$ is concentrated.  By Theorem~\ref{Thm:MainThm}, it follows that the splitting type of $D$ is $(e_1, \ldots , e_n)$.
\end{proof}

\bibliography{math}

\begin{thebibliography}{CPJ22b}

\bibitem[Bak08]{Baker08}
M.~Baker.
\newblock Specialization of linear systems from curves to graphs.
\newblock {\em Algebra Number Theory}, 2(6):613--653, 2008.

\bibitem[BN07]{BakerNorine07}
M.~Baker and S.~Norine.
\newblock {R}iemann-{R}och and {A}bel-{J}acobi theory on a finite graph.
\newblock {\em Adv. Math.}, 215(2):766--788, 2007.

\bibitem[Cil84]{Ciliberto84}
Ciro Ciliberto.
\newblock Some applications of a classical method of {C}astelnuovo.
\newblock In {\em Geometry seminars, 1982--1983 ({B}ologna, 1982/1983)}, pages
  17--43. Univ. Stud. Bologna, Bologna, 1984.

\bibitem[CLB16]{CoriLeBorgne}
Robert Cori and Yvan Le~Borgne.
\newblock On computation of {B}aker and {N}orine's rank on complete graphs.
\newblock {\em Electron. J. Combin.}, 23(1):Paper 1.31, 47, 2016.

\bibitem[CP17]{CoolsPanizzut}
F.~Cools and M.~Panizzut.
\newblock The gonality sequence of complete graphs.
\newblock {\em Electron. J. Combin.}, 24(4):Paper No. 4.1, 20, 2017.

\bibitem[CPJ22a]{CookPowellJensen22}
K.~Cook-Powell and D.~Jensen.
\newblock Components of {B}rill-{N}oether loci for curves with fixed gonality.
\newblock {\em Michigan Math. J.}, 71(1):19--45, 2022.

\bibitem[CPJ22b]{CPJ22}
Kaelin Cook-Powell and David Jensen.
\newblock Tropical methods in {H}urwitz-{B}rill-{N}oether theory.
\newblock {\em Adv. Math.}, 398:Paper No. 108199, 42, 2022.

\bibitem[FL81]{FultonLazarsfeld81}
W.~Fulton and R.~Lazarsfeld.
\newblock On the connectedness of degeneracy loci and special divisors.
\newblock {\em Acta Math.}, 146(3-4):271--283, 1981.

\bibitem[GH80]{GriffithsHarris80}
P.~Griffiths and J.~Harris.
\newblock On the variety of special linear systems on a general algebraic
  curve.
\newblock {\em Duke Math. J.}, 47(1):233--272, 1980.

\bibitem[Gie82]{Gieseker82}
D.~Gieseker.
\newblock Stable curves and special divisors: {P}etri's conjecture.
\newblock {\em Invent. Math.}, 66(2):251--275, 1982.

\bibitem[Har86]{Hartshorne86}
R.~Hartshorne.
\newblock Generalized divisors on {G}orenstein curves and a theorem of
  {N}oether.
\newblock {\em J. Math. Kyoto Univ.}, 26(3):375--386, 1986.

\bibitem[Lar21]{Larson21}
Hannah~K. Larson.
\newblock A refined {B}rill-{N}oether theory over {H}urwitz spaces.
\newblock {\em Invent. Math.}, 224(3):767--790, 2021.

\bibitem[LLV25]{LLV25}
Eric Larson, Hannah Larson, and Isabel Vogt.
\newblock Global {B}rill--{N}oether theory over the {H}urwitz space.
\newblock {\em Geom. Topol.}, 29(1):193--257, 2025.

\bibitem[LV24]{LarsonVemulapalli}
H.~Larson and S.~Vemulapalli.
\newblock Brill--{N}oether theory of smooth curves in the plane and on
  {H}irzebruch surfaces, 2024.

\bibitem[Noe82]{Noether82}
M.~Noether.
\newblock Zur {G}rundlegung der {T}heorie der algebraischen {R}aumcurven.
\newblock {\em J. Reine Angew. Math.}, 93:271--318, 1882.

\end{thebibliography}

\end{document}